\documentclass[11 pt]{article}
\usepackage{amsmath,amsthm,amssymb,amsxtra, graphicx}
\usepackage{color}
\textheight
23.0 cm  
\date{}
\newtheorem{theorem}{Theorem}[section]
\newtheorem{lemma}[theorem]{Lemma}
\newtheorem{proposition}[theorem]{Proposition}
\newtheorem{corollary}[theorem]{Corollary}

\newcommand{\ao}{\vec{A}}
\newcommand{\bo}{\vec{B}}
\newcommand{\co}{\vec{C}}

\newcommand{\xo}{\vec{X}}
\newcommand{\yo}{\vec{Y}}
\newcommand{\zo}{\vec{Z}}

\begin{document}
	
\setcounter{page}{1}
	
\title{\bf Interpolation of compact bilinear operators}
\author{Mieczys{\l}aw Masty{\l}o and Eduardo B.~Silva}

\date{}
\maketitle

\noindent
\renewcommand{\thefootnote}{\fnsymbol{footnote}}
\footnotetext{The first named author was supported by the National Science
Centre of Poland, project 2015/17/B/ST1/00064.}

\noindent

\begin{abstract}
\noindent
We investigate the stability of compactness of bilinear operators acting on the product of interpolation of Banach spaces.
~We develop a general framework for such results and our method applies to abstract methods of interpolation in the sense
of Aronszajn and Gagliardo.~A~key step is to show an~one-sided bilinear interpolation theorem on compactness for bilinear
operators on couples satisfying an approximation property. We show applications to general cases, including Peetre's
method and the general real interpolation methods.
\end{abstract}

\maketitle

\vspace{7 mm}

\section{Introduction}

In recent years various properties of bilinear and as multilinear operators are being studied intensively. Interest
in this study has increased since these operators are connected to important applications. We mention applications
in harmonic analysis in the study of $p$-Sidon sets (see \cite{Bayart}). Bilinear operators appear in applications
in elasticity. We point out the Newton--Kantorovi\v c effective method for solving certain equations involving
bilinear operators on Banach spaces (see \cite{AM}).~These operators also play an important role in scattering
theory (see \cite{Kupsch}). The bilinear interpolation theorems are powerful tools in the theory of Banach operator
ideals.

An important question related to the behavior of interpolation of compact operators is whether an operator acting
between Banach couples and  compactly on one or both of the ‘endpoint’ spaces, also acts compactly on the interpolation
spaces generated by the couples. It is a~natural question if there are variants of known results in the setting of
bilinear operators. In current paper, we discuss interpolation of bilinear compact operators. The problem
of interpolation of bilinear  operators by the classical real method was first studied by Lions and Peetre in their
seminal paper \cite{LP}. Calder\'on studied the same problem in his fundamental paper \cite{Cal} for the lower complex
method. In addition  the interpolation of compact bilinear operators is also considered in \cite[10.4]{Cal}. The counterpart
has been  studied recently in \cite{FS} for the real method $(\,\cdot\,)_{\rho, q}$ with quasi-power function parameter
$\rho$ and $1\leq q\leq \infty$, which is a~generalization of the classical real method generated by $\rho(t) = t^{\theta}$
for all $t>0$ with $\theta \in (0, 1)$. Results from \cite{FS} were extended in \cite{FM} for larger class of real methods
of interpolation.

The problem of interpolation of bilinear operators by abstract interpolation methods was studied in \cite{Mas1, Mas2}.
The stability of compactness of bilinear operators acting on the product of the real interpolation spaces  has been studied recently
as well as in \cite{Cobos1, Luz, MS}. We also mention that in a~very recent paper \cite{BC} the authors  established an interesting
formula for the  measure of non-compactness of bilinear operators interpolated by the general real method. In particular this
result applies to the real method and to the real method with a function parameter.

The investigation on compactness property of bilinear operators acting on the product of abstract interpolation Banach spaces is
not currently much advanced. In a recent paper \cite{MS}, interpolation of the measure of non-compactness of bilinear operators
is studied. In this paper the results of a~general nature are proved which states that, for a~large class of interpolation functors
preserving bilinear interpolation estimates of measures of non-compactness of interpolated linear operators between Banach couples,
can be lifted to bilinear operators. It has been shown that, as an application, the measure of non-compactness of bilinear operators
behave well under the real method of interpolation. Applications of these results comprise theorems on stability of compactness of
interpolated operators.

We point out that these results are proved for the class of bilinear operators $T$ defined on the products of intersections
$(X_0 \cap X_1) \times (Y_0 \cap Y_1)$ of Banach couples $(X_0, X_1)$ and $(Y_0,Y_1)$ with values for the intersection
$Z_0 \cap Z_1$ of a Banach couple $(Z_0, Z_1)$, such that, for both $j=0$ and $j=1$, we have
\[
\|T(x, y)\|_{Z_j} \leq M_j\|x\|_{X_j} \|y\|_{Y_j}, \quad\, (x, y) \in (X_0 \cap X_1) \times (Y_0 \cap Y_1).
\]
The study of abstract interpolation properties of this class of bilinear operators requires some natural restrictions
whenever we expect to prove an abstract general result.~It should be pointed out that many important bilinear operators
in harmonic analysis belong to the above type defined for a~special class of Banach function spaces.~We refer to \cite{BT}
and \cite{BC}, where compactness of commutators of bilinear Calder\'on--Zygmund operators and multiplication by functions
in $C\!M\!O$ of $B\!M\!O$ from the product $L_p \times L_q$ into $L_r$ is studied under the conditions $1<p, q <\infty$
and $1/p + 1/q = 1/r \leq 1$.

In this paper, we provide a~very general abstract approach in the study of the stability of compactness property of (bounded)
bilinear operators acting on products of abstract interpolation of Banach spaces. We consider bilinear operators
$T\colon (X_0 + X_1)\times (Y_0 + Y_1) \to Z_0 + Z_1$, such that the restriction $T\colon X_j \times Y_j \to Z_j$
is bounded for $j=0$ and $j=1$. We prove an~one-sided bilinear interpolation theorem on compactness for bilinear
operators of this type, acting on couples satisfying an approximation property $(Ap)$, introduced in a~remarkable paper by Cobos
and Peetre \cite{CP}. Result is lifted to the~wider class of abstract methods of interpolation in the sense of Aronszajn and
Gagliardo, allowing us to obtain a~very general compactness result for interpolation of bilinear operators.~As applications,
we consider the real, complex and Peetre interpolation methods.

\section{Definitions and preliminary results}

We use notations from Banach space theory. The (closed) unit ball of a~Banach space $X$ is denoted by $B_X$. As usual, we denote
by $\mathcal{L}(X, Y)$, the Banach space of all bounded operators $T\colon X\to Y$ from Banach space $X$ into $Y$, equipped
with uniform norm.

The product $X\times Y$ of two Banach spaces is equipped with the norm $\|(x, y)\|_{X\times Y} := \max\{\|x\|_X, \|y\|_Y\}$
for all $(x, y) \in X\times Y$. $\mathcal{L}_2(X\times Y, Z)$ denotes the Banach space of all $2$-linear bounded mappings
$T\colon X\times Y \to Z$, equipped with the norm
\[
\|T\|_{X\times Y\to Z} := \sup\big\{\|T(x, y)\|_{Z}; \, (x, y) \in B_{X\times Y}\big\}.
\]
Mapping $T \in \mathcal{L}_2(X\times Y, Z)$ is called a~\emph{bilinear operator}.

A~$2$-linear mapping $X\times Y\to Z$ is said to be compact if $T$ maps bounded subsets of $X\times Y$ into precompact subsets
of $Z$. This condition is equivalent to precompactness of $T(B_{X\times Y})$ in $Z$. We will use an equivalent condition,
namely for any bounded sequence $\{(x_n, y_n)\}$ in $X\times Y$, the sequence $\{T(x_n, y_n)\}$ has a~convergent subsequence
in $Z$. We refer to \cite{BT} for examples of bilinear compact operators.

If $S_0\colon X_0 \to Y_0$ and $S_1\colon X_1 \to Y_1$ are operators between Banach spaces, then we denote by $(S_0, S_1)$
the bounded linear operator from $X_0 \times X_1$ to $Y_0 \times Y_1$ defined by
\[
(S_0, S_1)(x_0, x_1): = (S_{0}x_0, S_{1}x_1),  \quad\, (x_0, x_1) \in X_{0} \times X_{1}.
\]

The following obvious proposition is required.

\begin{proposition}
\label{lift}
Let $S_0\colon X_0 \to Y_0$ and $S_1\colon X_1 \to Y_1$ be surjective operators between Banach spaces. Suppose that $W$ and $Z$
are Banach spaces and let $T \colon Y_0 \times Y_1 \to Z$ be a~bilinear operator. If $V\colon Z \to W$ is an isomorphic embedding,
then $T\colon Y_0 \times Y_1 \to Z$ is compact if, and only if, the bilinear operator $VT(S_0, S_1) \colon X_0 \times X_1 \to W$
is compact.
\end{proposition}

We will use standard notation from the interpolation theory. As a rule, we follow \cite{BL}. If $X$ is an intermediate Banach
space with respect to a~couple $\xo=(X_0, X_1)$, we let $X^{\circ}$ be the closed hull of $X_0 \cap X_1$ in $X$, and
the~Banach couple $(X_0^{\circ}, X_1^{\circ})$ is denoted by $\xo^{\circ}$. A~Banach couple $(X_0, X_1)$ is called \emph{regular}
if $X_j^{\circ}= X_j$ for $j\in \{0, 1\}$.

We shall recall that a~mapping $\mathcal{F}$ from the category of all couples of Banach spaces into the category of all
Banach spaces is said to be an interpolation \emph{functor} (or \emph{method}) if, for any couple $\xo:=(X_0, X_1)$,
the Banach space $\mathcal{F}(X_0, X_1)$ is  intermediate with respect to $\xo$ (i.e., $X_0 \cap X_1 \hookrightarrow \mathcal{F}(\xo)
\hookrightarrow X_0 + X_1$), and $T\colon \mathcal{F}(X_0, X_1) \to \mathcal{F}(Y_0, Y_1)$ for all  $T\colon (X_0, X_1) \to (Y_0, Y_1)$
As usual, the notation $T\colon (X_0, X_1) \to (Y_0, Y_1)$ means that $T\colon X_0 + X_1 \to Y_0 + Y_1$ is a~linear operator such that
the restrictions of $T$ to space $X_j$ is a~bounded operator from $X_j$ to $Y_j$, for both $j=0$ and $j=1$. The
interpolation functor is said to be exact if $\|T\|_{\mathcal{F}(X_0,X_1)\to \mathcal{F}(Y_0,Y_1)} \leq \max_{j=0, 1}
\|T|_{X_j}\|_{X_j\to Y_j}.$

The set of all functions $\varphi \colon (0, \infty) \times (0, \infty) \to (0, \infty)$, which are non-decreasing in each variable
and positively homogeneous (that is, $\varphi(\lambda s,  \lambda t) = \lambda \varphi(s, t)$ for all $\lambda, s, t>0$), is denoted
by $\Phi$. The subset of all $\varphi \in \Phi$, such that $\lim_{s\to 0}\varphi(1, s)= \lim_{s\to 0} \varphi(s, 1)=0$
is denoted by $\Phi_0$.

Note that for any $\varphi \in \Phi$, $\varphi \neq 0$, the function $(s, t)\mapsto 1/\varphi(1/s, 1/t)$ defined for all $s, t>0$
also belongs to $\Phi$. This function will be denoted by $\varphi^{*}$. Observe that functions from $\Phi$ are continuous by
monotonicity. Note that every $\varphi \in \Phi$ can be extended by continuity to $[0, \infty) \times [0,\infty)$. This extension
will be denoted by the same symbol $\varphi$. The simplest examples of interpolation functions are $as + bt$, $\max\{as, bt\}$,
and $\min\{as, bt\}$, where $a, b>0$ and the power functions $s^{1-\theta}t^{\theta}$, where $0\leq \theta \leq 1$.

Let $\xo=(X_0, X_1)$ be a Banach couple. For every $s,t>0$, we define the $K$-functional
\[
K(s,t, x; \xo) = \inf\{s\|x_0\|_{X_0} + t \|x_1\|_{X_1}; \, x= x_0 + x_1\}, \quad\, x\in X_0 + X_1.
\]
In the sequel, for $x\in X_0 + X_1$,
\[
K(t, x; \xo):= K(1, t, x; \xo), \quad\, t>0.
\]

For any Banach space $X$, such that $X_0 \cap X_1 \hookrightarrow X$ (resp., $X \hookrightarrow X_0 + X_1$),
we define (the fundamental function of $X$ with respect to $\xo$) $\phi_X \in \Phi$ (resp., $\psi_X \in \Phi$) by
\[
\phi_X(s, t) = \sup\{\|x\|_{X};\, x\in X_0 \cap X_{1},\,\,\|x\|_{X_0} \leq s, \|x\|_{X_1} \leq t\}
\]
(resp.,
\[
\psi_X(s,t; \xo):= \sup\{K(s, t, x; \xo); \,  \|x\|_X =1\}, \quad\, s, t>0).
\]
Let $\xo=(X_0, X_1)$, $\yo=(Y_0, Y_1)$ and $\zo=(Z_0, Z_1)$ be Banach couples. If an operator
$T\colon (X_0 + X_1)\times (Y_0 + Y_1) \to Z_0 + Z_1$ is such that the restrictions
$T\colon X_0 \times Y_0 \to Z_0$ and $T\colon X_1 \times Y_1 \to Z_1$ are bilinear operators, then we write
$T\colon \xo \times \yo \to \zo$.

Let's assume that $X$, $Y$ and $Z$ are Banach spaces intermediate to Banach couples $\xo$, $\yo$ and $\zo$,
respectively. If for every bilinear operator $T\colon \xo \times \yo \to \zo$, the restriction of $T$ is bounded
from $X\times Y$ to $Z$, then $X$, $Y$ and $Z$ are called \emph{bilinear interpolation spaces} with respect to
$(\xo, \yo)$ and $\zo$ ($(X, Y; Z) \in \mathcal{B}(\xo, \yo; \zo)$ for short). If in addition there exists
a~function $\varphi \in \Phi$, such that
\[
\|T\|_{X\times Y\to Z} \leq \varphi\big(\|T\|_{X_0 \times Y_0 \to Z_0}, \|T\|_{X_1 \times Y_1 \to Z_1}\big),
\]
then $X$, $Y$ and $Z$ are called $\varphi$-bilinear interpolation spaces, and we write
$(X, Y; Z) \in \mathcal{B}_{\varphi}(\xo, \yo; \zo)$ for short.

The following observation is required.

\begin{proposition}
\label{lambda}
Let $A, B$ and $C$ be Banach spaces and let $\{T_n\}_{n=1}^{\infty}$ be a~sequence of bilinear operators from
$A\times B$ to $C$, such that $\|T_n\|_{A \times B \to C} \to \lambda$ as $n\to \infty$. Then,
there exists a~sequence $\{(a_n, b_n)\}_{n=1}^{\infty}$ in the unit ball of $A\times B$, such that
\[
\lim_{n\to \infty} \|T_n(a_n, b_n)\|_C  = \lambda.
\]
\end{proposition}

We also quote the following technical result. Since the proof is obvious, it will be omitted.
\begin{proposition}
Let $(A_0, A_1)$ and $(B_0, B_1)$ be Banach couples and let $C$ be a~Banach space. Assume that
$T\colon (A_0 + A_1)\times (B_0 + B_1) \to C$ is $2$-linear mapping, such that $T$ is a~bilinear
operator from $A_j \times B_j$ to $C$ for $j\in \{0, 1\}$. Then $T\colon (A_0 + A_1) \times (B_0 \cap B_1) \to C$
and $T\colon (A_0 \cap A_1) \times (B_0 + B_1) \to C$ are bounded bilinear operators.
\end{proposition}

We will now provide variants of Lions--Peetre compactness results in the setting of bilinear operators.

\begin{lemma}
\label{LP1}
Let $A$ and $B$ be Banach spaces, $(C_0, C_1)$ be a~Banach couple and $C$ be a~Banach space, such that
$C_0 \cap C_1 \hookrightarrow C$. Assume that a~bilinear operator $T \colon A \times B \to C_0 \cap C_1$ is
such that $T\colon A\times B \to C_0$ is compact. Then $T\colon A\times B \to C$ is also compact whenever
$\phi_C(s, 1) \to 0$ as $s\to 0$.
\end{lemma}

\begin{proof}
Without loss of generality we may assume that $\max_{j=0, 1}\|T\|_{A\times B\to C_j} \leq 1$. Let $\{(a_n, b_n)\}$ be
a~bounded sequence in the unit ball of $A \times B$. Since $T$ is compact from $A \times B$ into $C_0$, by passing
to subsequence, we may assume that $\{T(a_n, b_n)\}$ is a~Cauchy sequence in $C_0$. Since
$\|T(a_n, b_n)- T(a_k, b_k)\|_{C_1} \leq 2$, it follows from monotonicity of function $\phi_X$ that for
each positive integer $n$ and $k$,
\begin{align*}
\|T(a_n, b_n)- T(a_k,b_k)\|_C & \leq \phi_C \big(\|T(a_n, b_n)- T(a_k,b_k)\|_{C_0}, \|T(a_n, b_n)- T(a_k, b_k)\|_{C_1}\big) \\
& \leq 2\phi_C \big(\|T(a_n, b_n)- T(a_k,b_k)\|_{C_0}, 1\big).
\end{align*}
Combining with our hypothesis that $\phi_C(s, 1) \to 0$ as $s\to 0+$ yields that $\{T(a_n, b_n)\}$ is a~Cauchy sequence in $C$.
The proof is complete.
\end{proof}

The next variant of the Lions--Peetre compactness result for bilinear operators is given in the following lemma.

\begin{lemma}
\label{LP2}
Let $C$ be any Banach space and $\ao=(A_0, A_1)$, $\bo=(B_0, B_1)$ be Banach couples and let $T\colon (A_0 + A_1) \times (B_0 + B_1) \to C$
be a bilinear operator, such that $T\in \mathcal{L}_2(A_j \times B_j, C)$ for $j\in \{0, 1\}$. Assume that $A\hookrightarrow A_0 + A_1$
and $B\hookrightarrow B_0 + B_1$ are Banach spaces, such that $\psi_A(t, 1) \to 0$, $\psi_B(t, 1) \to 0$ as $t \to 0$. Then, for any
compact bilinear operator $T\colon A_0\times B_0 \to C$, the restriction $T\colon A\times B \to C$ is also a~compact operator.
\end{lemma}

\begin{proof}
We may assume without loss of generality that the norms of the inclusion maps
$A\hookrightarrow A_0 + A_1$ and $B\hookrightarrow B_0 + B_1$ are less than or equal to $1$ and that
\begin{align*}
\|T\|_{(A_0+A_1) \times (B_0 + B_1)\to C} \leq 1, \quad\, \max_{j=0, 1}\|T\|_{A_j \times B_j\to C} \leq 1.
\end{align*}
Clearly this implies that $T\colon A \times B_1 \to Z$ and $T\colon A_0 \times B \to Z$ are bounded
bilinear operators with norms less than or equal to $1$.

To simplify notation, we put $\psi_0(t):= \psi_X(1, t)$ and $\psi_1(t):= \psi_Y(1, t)$ for all $t>0$.
Our hypothesis about limits is equivalent to
\begin{align*}
\lim_{t\to \infty} \frac{\psi_0(t)}{t} = 0, \quad\, \lim_{t\to \infty} \frac{\psi_1(t)}{t} = 0.
\end{align*}
Let's fix a~sequence $\{(a_n, b_n)\}$ in the unit ball of $A\times B$. The assumptions on limits allow us
to choose, for a~given $\varepsilon >0$, there exists a~sufficiently large $t>0$, such that
\[
\max\bigg\{\frac{\psi_0(t)}{t}, \, \frac{\psi_1(t)}{t},\,
\frac{\psi_0(t)}{t}\,\frac{\psi_1(t)}{t}\bigg\} < \frac{\varepsilon}{8}.
\]
From the definition of $\psi_X$ and $\psi_Y$, it follows that for all $n\in \mathbb{N}$ and
chosen $t>0$, we have
\[
K(t, a_n; \ao) \leq \psi_0(t), \quad\, K(t, b_n; \bo) \leq \psi_1(t).
\]
Then, for each $n\in \mathbb{N}$, we find the decompositions $a_n = a_n^{0} + a_n^{1}$ and
$b_n = b_n^{0} + b_{n}^1$ with $a_n^{j} \in A_j$ and $b_n^{j} \in B_j$ for each $j\in \{0,1\}$, such that
\begin{align*}
\|a_n^{0}\|_{A_0} + t \|a_n^{1}\|_{A_1} \leq 2 K(t, a_n; \ao), \quad\, \|b_n^{0}\|_{B_0} + t \|b_n^{1}\|_{B_1} \leq 2 K(t, a_n; \ao).
\end{align*}
The combination of these inequalities yields for each $n\in \mathbb{N}$,
\begin{align*}
\|a_n^{0}\|_{A_0} + t \|a_n^{1}\|_{A_1} \leq 2 \psi_0(t), \quad\, \|b_n^{0}\|_{B_0} + t \|b_n^{1}\|_{B_1} \leq 2 \psi_1(t).
\end{align*}
Hence $\{a_n^{0}\}$ and $\{b_n^0\}$ are bounded sequences in $A_0$ and $B_0$, respectively. Since
$T\colon A_0\times B_0 \to C$ is a~compact bilinear operator, by passing to a~subsequence, if necessary,
we may assume that there exists $z\in Z$, such that for some $N=N(\varepsilon)$,
\[
\|T(a_n^{0}, b_n^{0}) - z\|_C < \frac{\varepsilon}{4}, \quad\, n >N.
\]
We claim that $\{T(a_n, b_n)\}$ converges to $z$ in $C$. We may observe that
\[
T(a_n, b_n) = T(a_0^{n}, b_0^{n}) + T(a_0^{n}, b_1^{n}) + T(a_1^{n}, b_n), \quad\, n\in \mathbb{N}.
\]
In combination with the above estimates, for each $n>N$ we have:
\begin{align*}
\|T(a_n, b_n) - z\|_C & \leq \|T(a_0^{n}, b_0^{n})-z\|_C  + \|T(a_n^{0}, b_1^{n})\|_C + \|T(a_1^{n}, b_n)\|_C \\
& \leq \frac{\varepsilon}{4} + \|T(a_n, b_1^{n})\|_C + \|T(a_1^{n}, b_1^{n})\|_C + \|T(a_1^{n}, b_n)\|_C \\
& \leq \frac{\varepsilon}{4} + \|a_n\|_X \|b_1^{n}\|_C + \|a_1^{n}\|_{A_1}\|b_1^{n}\|_{B_1} + \|a_1^{n}\|_{A_1}\|b_n\|_C \\
& \leq \frac{\varepsilon}{4} + 2\frac{\psi_1(t)}{t} + 4 \frac{\psi_0(t)}{t} \frac{\psi_1(t)}{t} + 2 \frac{\psi_1(t)}{t}
\leq \varepsilon.
\end{align*}
This proves the claim and the proof is complete.
\end{proof}

\section{Interpolation of compact bilinear operators on couples with approximation property}

In this section we prove a~key one-sided compactness interpolation theorem for bilinear operators acting on Banach couples which
satisfies the approximation property $(Ap)$. Following \cite{CP}, we recall that a~Banach couple $\ao = (A_0, A_1)$
satisfies the approximation property $(Ap)$ if there is a~sequence $\{P_n\}_{n=1}^{\infty}$ of operators from $A_0 + A_1$ into
$A_0 \cap A_1$ and two other sequences $\{P^{+}_n\}$ and $\{P^{-}_n\}_{n=1}^{\infty}$, of operators from $A_0 + A_1$ into
$A_0 + A_1$, such that
\begin{itemize}
\item[{\rm(I)}] They are uniformly bounded in $\ao$, i.e.,
\[
C:=\sup_{n\in \mathbb{N}} \big\{\|P_n\|_{\ao \to \ao}\,, \|P_{n}^{+}\|_{\ao \to \ao}\,,\|P_n^ {-}\|_{\ao \to \ao}\big\} <\infty.
\]
\item[{\rm(II)}] The identity operator $I$ on $A_0 + A_1$ may be written as
\[
I=P_n + P_{n}^{+} + P_n^{-}, \quad\, n \in \mathbb{N}.
\]
\item[{\rm(III)}] For each $n\in \mathbb{N}$, we have $P^{+}\colon A_0 \to A_1$ and $P^{-}_n \colon A_1 \to A_0$, with
\[
\lim_{n\to \infty} \|P_n^{+}\|_{A_0 \to A_1} = \lim_{n\to \infty} \|P_n^{-}\|_{A_1 \to A_0} = 0.
\]
\end{itemize}

\begin{lemma}
\label{circle}
Suppose that a Banach couple $(A_0, A_1)$ satisfies the approximation property $(Ap)$. Then the following
holds{\rm:}
\begin{itemize}
\item[{\rm(i)}] If $a\in A_0^{\circ}$, then $\|P^{-}_n a\|_{A_0} \to 0$ as $n\to \infty${\rm;}
\item[{\rm(ii)}] If $a \in A_1^{\circ}$, then $\|P^{+}_n a\|_{A_1}\to 0$ as $n\to \infty$.
\end{itemize}
\end{lemma}

\vspace{2 mm}

The next theorem is the core for our main result in the following section.

\begin{theorem}
\label{core}
Let $A$, $B$ and $C$ be Banach spaces intermediate to Banach couples $\ao=(A_0, A_1)$, $\bo=(B_0, B_1)$ and
$\co=(C_0, C_1)$, respectively, which satisfy the approximation property $(Ap)$ and $\psi_A(s, 1) \to 0$,
$\psi_B(s, 1) \to 0$ and $\phi_C(s, 1) \to 0$ as $s\to 0$. Assume that $(A, B; C) \in \mathcal{B}_{\varphi}(\ao, \bo; \co)$
and $(A, B; C) \in \mathcal{B}_{\varphi}(\ao^{\circ}, \bo^{\circ}; \co)$ with $\varphi \in \Phi_0$.~Then,
for any bilinear operator $T\colon (A_0, A_1) \times (B_0, B_1) \to (C_0, C_1)$, such that
$T\colon A_0 \times B_0 \to C_0$ is compact, it follows that $T\colon A\times B \to C$ is also compact.
\end{theorem}

\begin{proof}
Let $\{P_{n}\}$, $\{P^{+}_{n}\}$, $\{P^{-}_{n}\}$, $\{Q_{n}\}$, $\{Q^{+}_{n}\}$, $\{Q^{-}_{n}\}$ and $\{R_{n}\}$,
$\{R^{+}_{n}\}$, $\{R^{-}_{n}\}$ be the corresponding approximating sequences in the Banach couples $\ao$, $\bo$ and $\co$,
respectively, satisfying the approximation property $(Ap)$.

To prove that $T\colon A \times B \rightarrow C$ is compact, we consider the following decomposition:
\begin{align*}
T & = T(P_n + P_{n}^+ + P_{n}^-,Q_n + Q_{n}^+ + Q_{n}^-) = T(P_n,Q_n) + T(P_n,Q_{n}^+) + T(P_n,Q_{n}^-) \\
&\,\,\,\,\,+ T(P_{n}^+,Q_n) + T(P_{n}^+,Q_n^+) + T(P_n^+,Q_n^-) + T(P_n^-,Q_n) + T(P_n^-,Q_n^+) + T(P_n^-,Q_n^-) \\
& =  T(P_n,Q_n) + T(P_n,Q_{n}^+) + (R_n + R_n^+ + R_n^-)T(P_n,Q_{n}^-) + T(P_{n}^+,Q_n) + T(P_{n}^+,Q_n^+) \\
&\,\,\,\,\,+ T(P_n^+,Q_n^-) + T(P_n^-,Q_n) + T(P_n^-,Q_n^+) + (R_n + R_n^+ + R_n^-)T(P_n^-,Q_n^-) \\
& = T(P_n,Q_n) + T(P_n,Q_{n}^+) + R_n T(P_n,Q_{n}^-) + R_n^+ T(P_n,Q_{n}^-) + R_n^- T(P_n,Q_{n}^-) \\
&\,\,\,\,\,+ T(P_{n}^+,Q_n) + T(P_{n}^+,Q_n^+) + T(P_n^+,Q_n^-) + T(P_n^-,Q_n) + T(P_n^-,Q_n^+) \\
&\,\,\,\,\,+ R_n T(P_n^-,Q_n^-) + R_n^+ T(P_n^-,Q_n^-) + R_n^- T(P_n^-,Q_n^-).
\end{align*}
We claim that each one of the bilinear operators: $T(P_n, Q_n)$, $R_n T(P_n,Q_n^{-})$, $R_n T(P_n^{-}, Q_n^{-})$,
$R_n^+ T(P_n,Q_{n}^{-})$ and $R_n^{+}T(P_n^{-}, Q_n^{-})\}$ are compact from $A \times B$ to $C$, for each $n\in \mathbb{N}$.

\vspace{2 mm}

\noindent
Several steps are required.

(i) We start with $T(P_n, Q_n)$ by using the following factorization for $j\in \{0,1\}$:
\[
T(P_n,Q_n) \colon A \times B \stackrel{(P_{n},Q_{n})}{\longrightarrow} (A_{0} \cap A_{1})
\times (B_{0} \cap B_{1}) \hookrightarrow A_{j} \times B_{j} \stackrel{T}{\longrightarrow} C_{j},
\]
Since $T\colon A_0 \times B_0 \to C_0$ is compact, it follows, by Lemma 2.4, that the bilinear operator
$T(P_{n},Q_{n})\colon A\times B \to C$ is compact.

(ii) Using the following factorization of $R_n T (P_{n}, Q^{-}_{n})$, for each $j\in \{0, 1\}$:
\[
R_n T (P_{n}, Q^{-}_{n})\colon A_j \times B_j \stackrel{T(P_{n},Q_{n}^-)}{\longrightarrow} C_j
\stackrel{R_n}{\longrightarrow} C_0 \cap C_1  \hookrightarrow C,
\]
we conclude, by Lemma \ref{LP2}, that $R_n T (P_{n}, Q^{-}_{n})\colon A \times B \rightarrow C$ is compact operator.

(iii) Now let us consider the following factorization for $j\in \{0,1\}$,
\[
R_{n}T(P_n^{-}, Q_n^{-})\colon A_j \times B_j \stackrel{(P_{n}^{-}, Q_{n}^{-})}{\longrightarrow} A_j
\times B_j \stackrel{T}{\longrightarrow} C_{j} \stackrel{R_n}{\longrightarrow} C_{0} \cap C_1 \hookrightarrow C.
\]
Since $T\colon A_0 \times B_0 \to C_0$, Lemma \ref{LP2} applies. Therefore,  $R_{n}T(P_{n}^{-}, Q_{n}^{-})$
is a~compact operator from $A \times B$ to $C$.

(iv) To show the compactness of $R_n^{+}T(P_{n}, Q^{-}_{n})$ for each $n$ and since $T\colon A_0 \times B_0 \to C_0$
is compact, we observe that
\[
T(P_{n}, Q_{n}^-) \colon A_0 \times B_0 \stackrel{(P_{n}, Q_{n}^-)}{\longrightarrow} A_0 \times B_0
\stackrel{T}{\longrightarrow} C_{0},
\]
is also compact. Since $R_{n}^{+}\colon (C_0, C_0) \to (C_0, C_1)$, $R_{n}^{+}\colon C_0 \to C_0 \cap C_1$ is a~bounded
operator for each $n\in \mathbb{N}$. Consequently, we conclude that $R_{n}^{+}T(P_{n}, Q^{-}_{n})\colon A_0 \times B_0
\to C_0 \cap C_1$ is a~compact operator compact. Then, by Lemma \ref{LP1}, it follows that
$R_{n}^{+}T(P_{n}, Q^{-}_{n})\colon A \times B \to C$ is also compact.

We show that all sequences of norms of bilinear operators from $A\times B$ to $C$ have limit equal to
$0$: $\{\|T(P_n, Q_{n}^{+})\|\}$, $\{\|R_n^{-}T(P_n, Q_{n}^{-})\|\}$, $\{\|T(P_{n}^{+}, Q_n)\|\}$,
$\{\|T(P_{n}^{+}, Q_n^{+})\|\}$, $\{\|T(P_n^{+}, Q_n^{-})\|\}$, $\{\|R_n^{-}T(P_n, Q_{n}^{-})\|\}$, $\{\|T(P_{n}^{+}, Q_n)\|\}$,
$\{\|T(P_{n}^{+}, Q_n^{+})\|\}$, $\{\|T(P_n^{+}, Q_n^{-})\|\}$, \linebreak
$\{\|T(P_n^{-}, Q_n)\|\}$, $\{\|T(P_n^{-}, Q_n^{+})\|\}$, $\{\|R_n^{-}T(P_n^{-}, Q_n^{-})\|\}$.

\vspace{1.5 mm}

(v) We prove that
\[
\lim_{n \to \infty} \|T(P_{n}, Q^{+}_{n})\|_{A \times B \to C} = 0.
\]
Observe that our hypothesis $(A, B; C) \in \mathcal{B}_{\varphi}(\ao, \bo; \co)$ yields
\begin{align*}
\|T(P_{n}, Q^{+}_{n})\|_{A \times B \to C} & \leq \varphi(\|T(P_{n}, Q^{+}_{n})\|_{A_0 \times B_0 \to C_0},
\|T(P_{n}, Q^{+}_{n})\|_{A_1 \times B_1 \to C_1}) \\
& \leq  C \varphi(\|T(P_{n}, Q^{+}_{n})||_{A_0 \times B_0 \to C_0}, 1),
\end{align*}
where $C \leq \max\{1, \sup_{n\geq 1}\|T(P_{n}, Q^{+}_{n})\|_{A_1 \times B_1 \to C_1}\}<\infty$. Thus, it is enough
to prove that $\|T(P_{n}, Q^{+}_{n})\|_{A_0 \times B_0 \to C_0} \to 0$ as $n \to \infty$. Suppose that this is false.
By passing  to a~subsequence, we may assume, without loss of generality, that for some $\lambda>0$,
\[
\lim_{n\to \infty} \|T(P_{n},Q^{+}_{n})\|_{A_0 \times B_0 \to C_0} = \lambda.
\]
It is clear that the sequence $\{T(P_{n}, Q^{+}_{n})\}$ of bilinear operators is uniformly bounded in
$A_{0} \times B_{0}$. Thus, Proposition \ref{lambda} shows that, by passing to a~subsequence, if necessary,
we may assume, without loss of generality, that there exists a~sequence $\{(a_n, b_n)\}$ in the
unit ball of $A_0 \times B_0$, such that $\|T(P_{n}, Q^{+}_{n})\|_{A_0 \times B_0 \to C_0} \to \lambda$ as
$n\to \infty$ and
\[
\lim_{n \to \infty} \|T(P_{n} a_{n}, Q^{+}_{n}b_{n})\|_{C_{0}}= \lambda.
\]
Since $T\colon A_0 \times B_0 \to C_0$ is compact, by passing a~subsequence, if necessary, we may assume
that $\{T(P_{n}a_{n}, Q^{+}_{n} b_{n})\}$ converges to some element $b$ in $C_0$ with $\|b\|_{C_0} = \lambda$.
We now observe that we have with $K:=\|T\|_{(A_0 + A_1)\times (B_0 + B_1)\,\to C_0 + C_1}$
\begin{align*}
\|T(P_{n} a_{n}, \; Q^{+}_{n} b_{n})\|_{C_0 + C_1} &
\leq K\,\|P_{n}a_{n}\|_{A_{0} + A_{1}} \|Q^{+}_{n} b_{n}\|_{B_0 + B_1} \\
& \leq K \,\|P_{n} a_{n}\|_{A_{0}} ||Q^{+}_{n}b_{n}\|_{B_{1}} \\
& \leq K \,\|P_{n}\|_{A_0 \to B_0}\|a_{n}\|_{A_0}\|Q^{+}_{n}\|_{B_0 \to B_1}\|b_{n}\|_{B_0}.
\end{align*}
Finally, note that $\|Q^{+}_{n}\|_{B_0 \to B_1} \to 0$ as $n \to \infty$ implies
$T(P_{n} a_{n}, Q^{+}_{n} b_{n}) \to 0$ in $C_0 + C_1$ as $n \to \infty$. Hence $b = 0$ and so $\lambda = 0$,
which is a~contradiction.

(vi) Our next task is to prove that $\|R_n^{-}T (P_{n}, Q^{-}_{n})\|_{A\times B \to C}\to 0$ as $n\to \infty$.
Similarly, we have
\begin{align*}
\|R_n^{-}T(P_{n}, Q_n^{-})\|_{A \times B \to C} & \leq K\,\varphi\big(\|R_n^- T(P_{n}, Q_n^{-})\|_{A_0 \times B_0 \to C_0},
\|R_n^{-}T(P_{n}, Q_n^{-})\|_{A_1 \times B_1 \to C_1}\big) \\
& \leq \widetilde{K} \varphi\big(\|R_n^{-}T(P_{n},Q^{-}_{n})\|_{A_0 \times B_0 \to C_0}, 1\big),
\end{align*}
for some constant $\widetilde{K}>0$ independent of $n$. Since the sequences of operators $\{P_n\}$, $\{Q_n^{-}\}$ and
$\{R_n^{-}\}$ are uniformly bounded, it is enough to prove that
\[
\lim_{n\to \infty} \|R_n^{-}T(P_n,Q_n^{-})\|_{A_0 \times B_0 \to C_0} = 0.
\]
Suppose, by a~contradiction, that this is not true. Then passing to a~subsequence, if necessary, we may assume that for
some $\delta>0$
\[
\lim_{n\to \infty} \|R_n^{-}T (P_{n},Q_{n}^{-})\|_{A_0 \times B_0 \to C_0} = \delta.
\]
Applying Proposition \ref{lambda}, we conclude that there exists a~sequence $\{(a_n, b_n)\}$ in the unit ball of
$A_0 \times B_0$ with $a_n \in A_0 \cap A_0$ and $b_n \in B_0 \cap B_1$, such that
\[
\lim_{n\to \infty} \|R_{n}^{-}T (P_{n}a_{n}, Q_{n}^{-}b_{n})\|_{C_0} = \delta.
\]
Since the sequence $\{(P_{n}a_{n}, Q_{n}^{-}b_{n})\}$ is bounded in $A_0 \times B_0$ and $T\colon A_0 \times B_0 \to C_0$
is compact by passing, if necessary, to a~subsequence, we may assume that the sequence $\{T(P_{n}a_n, Q_{n}^{-}b_n)\}$
converges in $C_0$ to some $c$. Thus, for $N$ large enough, we have
\[
\|R_{n}^{-}c\|_{C_0} > \frac{\delta}{2}, \quad\, n>N.
\]
Observe that $(a_n, b_n) \in (A_0 \cap A_1) \times (B_0 \cap B_1)$ implies $\{T(P_{n}a_{n}, Q_{n}^{-}b_{n})\} \subset C_0 \cap C_1$
for each $n$. Recall that $c= \lim_{n \to \infty} T(P_{n}a_{n}, Q_{n}^{-}b_{n})$ in $C_0$ and so $c \in C_0^{\circ}$.
Then, by Lemma \ref{circle},
\[
\lim_{n \to \infty}\|R_{n}^{-}c\|_{C_0} = 0,
\]
which is a contradiction with the above estimate.

(vii) To prove that
\begin{align*}
\lim_{n \to \infty} \|T(P_n^{+}, Q_n)\|_{A\times B \to C} & = \lim_{n \to \infty} \|T(P_n^{+}, Q_n^{+})\|_{A\times B \to C}
= \lim_{n \to \infty} \|T(P_n^{+}, Q_n^{-})\|_{A\times B \to C} \\
& = \lim_{n \to \infty} \|T(P_n^{-}, Q_n)\|_{A\times B\to C} = \lim_{n \to \infty} \|T(P_n^{-}, Q_n^{+})\|_{A\times B\to C} = 0,
\end{align*}
we use our hypothesis $\varphi \in \Phi_0$ and we proceed similarly as in (v).

(viii) For the sequence $\{R_n^{+}T(P_{n}^{-}, Q_n^{-})\}$, we have
\begin{align*}
\|R_n^{+}T(P_{n}^{-}, Q_n^{-})\|_{A \times B \to C} & \leq M
\varphi(\|R_n^{+}T(P_{n}^{-}, Q^{-}_{n})\|_{A_0 \times B_0 \to C_0}, \|R_n^{+}T(P_{n}^{-}, Q^{-}_{n})\|_{A_1 \times B_1 \to C_1}) \\
& \leq M \varphi(1, \|R_n^{+}T(P_n^{-}, Q_n^{-})\|_{A_1 \times B_1 \to C_1}),
\end{align*}
where $M:= \max\big\{1, \sup_{n\geq 1} \|R_n^{+}T(P_{n}^{-}, Q^{-}_{n})\|_{A_0 \times B_0 \to C_0}\big\}< \infty$. By approximation
property $(Ap)$, the sequences of operators $\{P_n^{-}\}$, $\{Q_n^{-}\}$ and $\{R_n^{+}\}$ are uniformly bounded. Thus it is sufficient
to show that $\|R_n^{+}T (P_n^{-}, Q_n^{-})\|_{A_1 \times B_1 \to C_1} \to 0$ as $n\to \infty$. Consequently, we combine the factorization
\begin{align*}
R_n^{+}T (P_n^{-}, Q_n^{-})\colon A_1 \times B_1 \stackrel{(P_{n}^-,Q_{n}^{-})}{\longrightarrow} A_0 \times B_0
\stackrel{T}{\longrightarrow} C_0 \stackrel{R_{n}^{+}}{\longrightarrow} C_{1},
\end{align*}
with the estimate
\begin{align*}
\|R_n^{+}T(P_n^{-}, Q_n^{-})\|_{A_1 \times B_1\to C_1} \leq \|R_n^{+}\|_{C_0 \to C_1}\,
\|T\|_{A_0 \times B_0 \to C_0}\|P_n^{-}\|_{A_1 \to A_0}\, \|Q_n^{-}\|_{B_1 \to B_0},
\end{align*}
to deduce, by the approximation property $(Ap)$, that
\[
\lim_{n\to \infty} \|R_n^{+}T(P_{n}^{-}, Q^{-}_{n})\|_{A_1 \times B_1\to C_1} = 0.
\]
We proceed similarly as in the (v) to obtain $\|R_n^{-}T(P_{n}^{-}, Q^{-}_{n})\|_{A\times B\to C} \to 0$ as $n\to \infty$.
The proof is complete.
\end{proof}

\section{Bilinear compactness theorem for Aronszajn--Gagliardo functors}

In this section we apply our results to bilinear operators acting from the product of interpolation spaces generated
by orbit functors to Banach spaces generated by coorbit functors, in the sense of Aronszajn-Gagliardo. We recall two
important constructions of the abstract interpolation theory by Aronszajn and Gagliardo \cite{AG}.

As usual, for non-empty set $\Gamma$ and any Banach space $X$, we denote by $\ell_1(\Gamma, X)$ (resp., $\ell_{\infty}(\Gamma, X)$
the Banach space of all absolutely summable (resp., bounded) families $\{x_{\gamma}\}_{\gamma \in \Gamma}$ of elements of $X$
indexed by $\Gamma$ and equipped with the norm
\[
\|\{x_\gamma\}\|_{\ell_1(\Gamma, X)} = \sum_{\gamma \in \Gamma} \|x_\gamma\|_{X}
\]
\big(resp.,
\[
\|\{x_\gamma\}\|_{\ell_\infty(\Gamma, X)} = \sup_{\gamma \in \Gamma} \|x_\gamma\|_{X}\big).
\]

Let $\ao = (A_0, A_1)$ be a fixed Banach couple and let $A$ be a~fixed intermediate space with respect to $\ao$.
If $\xo = (X_0, X_1)$ is any Banach couple and $\Gamma:=B_{L(\ao, \xo)}$ is the unit ball of the Banach space
$L(\ao, \xo)$, then we define a~mapping $\pi_{\ao}\colon \ell_1(\Gamma, A_0 + A_1) \to X_0 + X_1$ by the formula,
\[
\pi_{\ao}\{a_S\} = \sum_{S\in \Gamma} S(a_S), \quad\, \{a_S\}\in \ell_1(\Gamma, A_0 + A_1).
\]
If the unit ball $B_{L(\xo, \ao)}$ of the Banach space $L(\xo, \ao)$ is denoted by $J$, for short, we also define
an operator $i_{\ao}\colon X_0 + X_1 \to \ell_\infty(J,A_0 + A_1)$ by the formula
\[
i_{\ao}x = \{Tx\}_{T\in J}, \quad\, x\in X_{0} + X_{1}.
\]
The Aronszajn--Gagliardo exact interpolation functors are defined by
\[
Orb^{\ao}_A(\xo) := \big\{x \in X_0 + X_1; \, x= \pi_{\ao}\{a_S\} \,\,\,\,\text{for some $\{a_S\} \in \ell_1(J, A)$} \big\}
\]
and
\[
Corb^{\ao}_A(\xo) := \big\{x\in X_0 + X_{1};\, i_{\ao}x \in \ell_{\infty}(J, A)\big\}.
\]
Norms in these Banach spaces are given by
\[
\|x\|_{Orb^{\ao}_A(\xo)} := \inf\Big\{\sum_{S\in J} \|a_S\|_{A}; \, x = \pi_{\ao}\{a_S\}\Big\}.
\]
and, respectively
\[
\|x\|_{Corb^{\ao}_A(\xo)} = \|i_{\ao}x\|_{\ell_{\infty}(J, A)} = \sup\big\{\|Tx\|_A;\, \|T\|_{\xo \to \ao} \leq 1\big\}.
\]
For simplicity, we often write $G^{\ao}_A$  (resp., $H^{\ao}_A$) instead of $Orb^{\ao}_A$ (resp., $Corb^{\ao}_A$).
Note that $G^{\ao}_A$ is the minimal interpolation functor satisfying $A\hookrightarrow G^{\ao}_A(\ao)$ and $H^{\ao}_A$ is
the maximal interpolation functor satisfying $H^{\ao}_A(\ao) \hookrightarrow A$.

We will use the following result.

\begin{theorem}
\label{lift}
Assume that $(A, B; C) \in \mathcal{B}_{\varphi}(\ao, \bo; \co)$ with $\varphi \in \Phi$. Then, for any Banach couples
$\xo= (X_0, X_1)$, $\yo= (Y_0, Y_1)$ and $\zo=(Z_0, Z_1)$,
\[
\big(Orb_{A}^{\ao}(\xo), Orb_{B}^{\bo}(\yo); Corb_{C}^{\co}(\zo)\big) \in \mathcal{B}_{\varphi} (\xo, \yo; \zo).
\]
\end{theorem}

\begin{proof}
Fix $T\colon \xo \times \yo \to \zo$. Assume that $(x, y)\in G_{A}^{\ao}(\xo) \times G_{B}^{\bo}(\yo)$, where
$x=Ua$ and $y=Vb$ with $(a, b)\in A\times B$, $U\colon \ao \to \xo$ and $V\colon \bo \to \yo$. Then, for a~given
operator $R\colon \zo \to \co$ with $\|R\|_{\zo \to \co}\leq 1$, we define a~bilinear operator
$S_{R}\colon (A_0 + A_1)\times (B_0 + B_1) \to C_0 + C_1$ by $S_{R}:= RT(U,V)$. Observe that for each $j\in \{0, 1\}$
and every $(a_j, b_j) \in A_j \times B_j$,
\begin{align*}
\|S_R(a_j, b_j)\|_{C_j}& \leq
\|R\|_{\zo\to \co}\,\|T\|_{X_j\times Y_j \to Z_j}\|U\|_{\ao\to \xo} \|V\|_{\bo\to \yo}\|a_j\|_{A_j}\,\|b_j\|_{B_j}\\
& \leq \|U\|_{\ao\to \xo} \|V\|_{\bo\to \yo}\,\|T\|_{X_j\times Y_j \to Z_j} \|a_j\|_{A_j}\,\|b_j\|_{B_j}.
\end{align*}
This implies that $S_{R}\colon \ao \times \bo \to \co$ with
\[
\|S_{R}\|_{\ao \times \bo \to \co} \leq \|U\|_{\ao\to \xo} \|V\|_{\bo\to \yo}\max_{j=0, 1}\|T\|_{X_j\times Y_j \to Z_j}.
\]
Thus, it follows, by our hypothesis $(A, B; C) \in \mathcal{B}_{\varphi}(\ao, \bo; \co)$, that
\[
\|S_{R}\|_{A\times B\to C} \leq \varphi(\|T\|_{X_0\times Y_0\to Z_0}, \|T\|_{X_1\times Y_1 \to Z_1})\|U\|_{\ao\to \xo} \|V\|_{\bo\to \yo}.
\]
and so,
\begin{align*}
\sup_{\|R\|_{\zo \to \co} \leq 1} \|R(T(x, y))\|_C & = \sup_{\|R\|_{\zo \to \co}\leq 1} \|R(T(Ua, Vb))\|_C =
\sup_{\|R\|_{\zo \to \co} \leq 1} \|S_{R}(a, b)\|_C \\
&\leq  \varphi(\|T\|_{X_0\times Y_0\to Z_0}, \|T\|_{X_1\times Y_1 \to Z_1})\|U\|_{\ao\to \xo} \|V\|_{\bo\to \yo}\|a\|_{A}\,\|b\|_{B}.
\end{align*}
This proves that $T(x, y) = T(Ua, Vb) \in H_{C}^{\co}(\zo)$ with
\[
\tag {$*$} \|T(x,y)\|_{H_{C}^{\co}(\zo)} \leq \varphi(\|T\|_{X_0\times Y_0\to Z_0}, \|T\|_{X_1 \times Y_1 \to Z_1})\|U\|_{\ao\to \xo} \|V\|_{\bo\to \yo}\|a\|_{A}\,\|b\|_{B}.
\]
We now assume that $(x, y) \in G_{A}^{\ao}(\xo) \times G_{B}^{\bo}(\yo)$ and consider arbitrary representations
\[
x= \sum_{i=1}^{\infty} U_{i} a_i, \quad\, y= \sum_{j=1}^{\infty} V_{j}b_j
\]
with
\[
\sum_{i=1}^{\infty} \|U_i\|_{\ao \to \xo} \|a_i\|_{A}<\infty, \quad\, \sum_{j=1}^{\infty} \|V_j\|_{\bo \to \yo}\|b_j\|_{B} <\infty.
\]
It is clear that the above series converge into $A_0 + A_1$ and $B_0 + B_1$, respectively. It follows (since $T$ is a~bilinear
operator from $(A_0 + A_1)\times (B_0 + B_1)$ to $C_0 + C_1$) that the following double series converges into $C_0 + C_1$ to $T(x, y)$,
\begin{align*}
T(x, y) = \sum_{i=1}^{\infty} \sum_{j=1}^{\infty} T(U_ia_i, V_jb_j).
\end{align*}
Applying estimate $(*)$, we obtain
\begin{align*}
& \|T(x, y)\|_{H_{C}^{\co}(\zo)} \leq \sum_{i=1}^{\infty} \sum_{j=1}^{\infty}\| T(U_ia_i, V_jb_j)\|_{H_{C}^{\co}(\zo)} \\
& \leq \varphi(\|T\|_{X_0\times Y_0\to Z_0}, \|T\|_{X_1\times Y_1 \to Z_1})\sum_{i=1}^{\infty} \sum_{j=1}^{\infty}
\|U_i\|_{\ao\to \xo} \|V_j\|_{\bo\to \yo}\|a_i\|_{A}\,\|b_j\|_{B}\\
& = \varphi(\|T\|_{X_0\times Y_0\to Z_0}, \|T\|_{X_1 \times Y_1 \to Z_1}) \Big(\sum_{i=1}^{\infty} \|U_i\|_{\ao \to \xo} \|a_i\|_{A}\Big)
\Big(\sum_{j=1}^{\infty} \|V_j\|_{\bo \to \yo}\|b_j\|_{B}\Big).
\end{align*}
Since the representations of $x\in G_{A}^{\ao}(\xo)$ and $y\in G_{B}^{\bo}(\yo)$ are arbitrary, we conclude that
$T\colon G_{A}^{\ao}(\xo)\times  G_{B}^{\bo}(\yo)\to H_{C}^{\co}(\zo)$ with
\[
\|T(x, y)\|_{H_{C}^{\co}(\zo)} \leq \varphi(\|T\|_{X_0\times Y_0\to Z_0}, \|T\|_{X_1 \times Y_1 \to Z_1})\|x\|_{G_{A}^{\ao}(\xo)}\,
\|y\|_{G_{B}^{\bo}(\yo)}.
\]
This proves that $(G_{A}^{\ao}(\xo), G_{B}^{\bo}(\yo); H_{C}^{\co}(\zo)) \in \mathcal{B}_{\varphi}(\xo, \yo; \zo)$, as required.
\end{proof}

To prove the main result of this section, we use the continuous inclusions from \cite[Lemmas 2.1 and 3.1]{CP}.
We state these inclusions for the sake of completeness and convenience of the readers.

\begin{lemma}
\label{inclusions}
Let $I$ be any non-empty set and let $A$ be a~Banach space intermediate with respect to a Banach couple
$\ao=(A_0, A_1)$. Then the following continuous inclusions hold{\rm:}
\[
\ell_1(I, G_{A}^{\ao}(A_0, A_1)) \hookrightarrow G_{A}^{\ao}(\ell_1(I, A_0), \ell_1(I, A_1)),
\]
\[
H_{A}^{\ao}(\ell_\infty(I, A_0), \ell_\infty(I, A_1)) \hookrightarrow \ell_\infty(I, H_{A}^{\ao}(A_0, A_1)).
\]
with norm less than or equal to $1$.
\end{lemma}

Following \cite{Ov}, the function $\varphi_{\mathcal{F}}$, which corresponds to an exact
interpolation functor $\mathcal{F}$ by the equality
\[
\mathcal{F}(s\mathbb{R}, t\mathbb{R}) = \varphi_{\mathcal{F}}(s,t)\mathbb{R},
\]
is called the \emph{characteristic function} of the functor $\mathcal{F}$. Here $\alpha \mathbb{R}$ denotes $\mathbb{R}$
equipped with the norm $\|\cdot\|_{\alpha \mathbb{R}} = \alpha |\cdot|$ for $\alpha >0$. We notice that $\varphi\in\Phi$.

\vspace{2 mm}

We omit the simple proof of the following technical fact.

\begin{proposition}
\label{fundorbit}
Let $A$ be an intermediate Banach space with respect to a couple $\ao$ of Banach spaces. Then the characteristic
function $\varphi_G$ of an exact interpolation functor $G:= G_{A}^{\ao}$ is given by
\[
\varphi_G(s, t) = (\psi_A)^{*}(s, t), \quad\, s, t>0.
\]
\end{proposition}

\vspace{2 mm}

We will need also the following lemma.

\begin{lemma}
\label{functionest}
If $\varphi_\mathcal{F}$ is a~characteristic function of an exact interpolation functor $\mathcal{F}$, then
\[
(\varphi_{\mathcal{F}})^{*}(s, t) = \sup_{\ao \in \vec{\mathcal{B}}} \psi_{\mathcal{F}(\ao)}(s, t), \quad\, s, t>0,
\]
where $\vec{\mathcal{B}}$ denotes the class of all Banach couples.
\end{lemma}

\begin{proof}
Let us fix a~Banach couple $\ao$. Then from the minimality property of an orbit functor $G:=G_{A}^{\ao}$
with $A := \mathcal{F}(\ao)$, it follows that, for any Banach couple $\xo$,
\[
G(\xo) \hookrightarrow \mathcal{F}(\xo)
\]
with the norm of the continuous inclusion map less than or equal to $1$. In particular this implies that
$\varphi_{G} \geq \varphi_{\mathcal{F}}$, and so,
\[
(\varphi_G)^{*}(s, t) \leq (\varphi_{\mathcal{F}})^{*}(s, t), \quad\, s, t>0.
\]
Since the characteristic function $\varphi_G$ of the functor $G$ satisfies $\varphi_G(s, t) = (\psi_A)^{*}(s, t)$
for all $s, t>0$, we conclude that
\[
\sup_{\ao \in \vec{\mathcal{B}}}\psi_{\mathcal{F}(\ao)}(s, t) \leq (\varphi_{\mathcal{F}})^{*}(s, t), \quad\, s, t>0.
\]
A~direct computation shows that, for fixed $u, v>0$ and all $\alpha \in \mathbb{R}$,
\[
K(s, t, \alpha; (u\mathbb{R}, v\mathbb{R})) = \min\{su, tv\}|\alpha|, \quad\, s, t>0.
\]
Hence, for all $s, t>0$, we get
\begin{align*}
\sup_{\ao \in \vec{\mathcal{B}}}\psi_{\mathcal{F}(\ao)}(s, t) & \geq \sup_{u, v>0}
\psi_{\mathcal{F}(u\mathbb{R}, v\mathbb{R})}(s, t) \\
& = \sup_{u, v>0} \, \sup_{\alpha\neq 0} \frac{K(s, t, \alpha;
(u\mathbb{R}, v\mathbb{R}))}{\|\alpha\|_{\mathcal{F}(u\mathbb{R}, v\mathbb{R})}}\\
&= \sup_{u, v>0}\frac{\min\{su, tv\}}{\varphi_\mathcal{F}(s, t)} = (\varphi_F)^{*}(s, t).
\end{align*}
This completes the proof.
\end{proof}

\begin{corollary}
\label{Xestimate}
Let $\ao = (A_0, A_1)$ be a fixed Banach couple and let $A$ be a~fixed intermediate space with respect to $\ao$. Then, for any
Banach couple $\xo$, the fundamental function $\psi_X$ of $X:=Orb_{A}^{\ao}(\xo)$ satisfies the estimate
\[
\psi_X(s, t) \leq \psi_A(s, t), \quad\, s, t >0.
\]
In particular $\psi_A \in \Phi_0$ implies $\psi_X \in \Phi_0$.
\end{corollary}

\begin{proof}
As mentioned in the proof of Lemma \ref{functionest}, the fundamental function of the functor $\mathcal{F}:=Orb_{A}^{\ao}$
satisfies $(\varphi_F)^{*} = \psi_{A}.$ Applying Lemma \ref{functionest} to the functor $\mathcal{F}$, the required estimate
follows (by $X=\mathcal{F}(\xo)$)
\[
\psi_X(s, t) = \sup\{K(s, t, x; \xo);\, \|x\|_X = 1\} \leq (\varphi_{\mathcal{F}})^{*}(s, t) = \psi_A(s, t), \quad\, s, t>0.
\]
\end{proof}

We now state and prove the following key theorem, which will be used repeatedly.

\begin{theorem}
\label{orbcorb}
Suppose that Banach spaces $A$, $B$, $C$ and  Banach couples $\ao = (A_0, A_1)$, $\bo=(B_0, B_1)$, $\co= (C_0, C_1)$
satisfy the conditions of Theorem $\ref{core}$. Let $(X_0, X_1)$, $(Y_0, Y_1)$, $(Z_0, Z_1)$ be any Banach couples and let
$T\colon (X_0, X_1) \times (Y_0, Y_1) \to (Z_0, Z_1)$ be a~bilinear operator, such that the restriction
$T\colon X_0 \times Y_0 \to Z_0$ is compact. Then
\[
T\colon Orb_{A}^{\ao}(X_0, X_1) \times Orb_{B}^{\bo}(Y_0, Y_1) \to Corb_{C}^{\co}(Z_0, Z_1)
\]
is also a~compact operator.
\end{theorem}

\begin{proof}
For simplicity, we denote by $I_0$, $I_1$ and $I$ the balls of the Banach spaces $L(\ao, \xo)$, $L(\bo, \yo)$
and $L(\co, \zo)$, respectively. From Proposition \ref{lift} combined with the definitions of minimal and maximal
interpolation functors, it follows that the operator
\[
T\colon G_{A}^{\ao}(X_0, X_1) \times  G_{B}^{\bo}(Y_0, Y_1) \to H_{C}^{\co}(Z_0, Z_1)
\]
is compact, if and only if, the operator $\widetilde{T}$ given by the formula
\[
\widetilde{T} := i_{\co}T(\pi_{\xo}, \pi_{\yo}) \colon \ell_1(I_0, A) \times \ell_1(I_1, B) \to \ell_{\infty}(I, C)
\]
is compact. We thus have the following factorizations for the restrictions
\[
\widetilde{T}\colon \ell_1(I_0, A_0) \times \ell_1(I_1, B_0) \stackrel{\pi_{\ao}} \longrightarrow X_0 \times Y_0
\stackrel{T} \longrightarrow Z_0 \stackrel{i_{\co}} \longrightarrow \ell_{\infty}(I, C_0),
\]
\[
\widetilde{T}\colon \ell_1(I_0, A_1) \times \ell_1(I_1, B_1) \stackrel{\pi_{\ao}} \longrightarrow X_1 \times Y_1
\stackrel{T} \longrightarrow Z_1 \stackrel{i_{\co}} \longrightarrow \ell_{\infty}(I, C_1).
\]
Applying Lemma \ref{lift}, we conclude that
\[
\widetilde{T}\colon G_{A}^{\ao}(\ell_1(I_0, A_0), \ell_1(I_0, A_1))
\times G_{B}^{\bo}(\ell_1(I_1, B_0),  \ell_1(I_1, B_1)) \to H_{C}^{\co}(\ell_\infty(I, C_0), \ell_\infty(I, C_1)).
\]
Since the couples $(A_0, A_1)$,  $(B_0, B_1)$ and $(C_0, C_1)$ have the approximation property $(Ap)$, then
the couples $(\ell_1(I_0, A_0), \ell_1(I_0, A_1))$, $(\ell_1(I_1, B_0),  \ell_1(I_1, B_1))$ and $(\ell_1(I, C_0),
\ell_1(I, C_1))$  inherit the same property. Further, our hypothesis $A_0=A_0^{\circ}$ and $B_0=B_0^{\circ}$
implies that, for couples $(\ell_1(I_0, A_0), \ell_1(I_0, A_1))$, $(\ell_1(I_1, B_0),  \ell_1(I_1, B_1))$, we have
\[
\ell_1(I_0, A_0)= \ell_1(I_0, A_0)^{\circ},\quad\, \ell_1(I_1, B_0)=\ell_1(I_1, B_0)^{\circ}.
\]
Combining the above facts, we conclude from Theorem \ref{core} that
\[
\widetilde{T}\colon G_{A}^{\ao}(\ell_1(I_0, A_0), \ell_1(I_0, A_1))
\times G_{B}^{\bo}(\ell_1(I_1, B_0),  \ell_1(I_1, B_1)) \to H_{C}^{\co}(\ell_\infty(I, C_0), \ell_\infty(I, C_1))
\]
is a~compact operator. Consequently, it follows, from Lemma \ref{inclusions},
\[
\widetilde{T}\colon \ell_1(I_0, G_{A}^{\ao}(A_0, A_1))\times  \ell_1(I_1, G_{A}^{\ao}(A_0, A_1))
\to \ell_\infty(I, H_{C}^{\co}(C_0, C_1))
\]
is compact. Combining this fact with obvious continuous inclusions
\[
\ell_1(I_0, A) \hookrightarrow G_{A}^{\ao}(A_0, A_1), \quad\, \ell_1(I_1, B) \hookrightarrow G_{B}^{\ao}(B_0, B_1), \quad\,
\ell_\infty(I, H_{C}^{\co}(C_0, C_1)) \hookrightarrow \ell_\infty(I, C)
\]
we conclude that
\[
\widetilde{T}\colon \ell_1(I_0, A) \times \ell_1(I_1, B) \to \ell_\infty(I, C)
\]
is a compact operator as required. The proof is complete.
\end{proof}

We conclude this section by specializing Theorem \ref{orbcorb} to specific couples $\ao$, $\bo$ and $\co$ satisfying the
approximation property $(Ap)$ and triples $(A, B; C) \in \mathcal{B}_{\varphi}(\ao, \bo\; \co)$ to get new results on
interpolation of bilinear compact operators. We start with applications which involve the Calder\'on complex method of
interpolation $[\,\cdot\,]_{\theta}$ with $\theta \in (0, 1)$. Information on this method is found in \cite{Cal, BL}.

\begin{theorem}
\label{complex}
Suppose that Banach couples $\ao$, $\bo$ and $\co$ satisfy the approximation property $(Ap)$. Then for any Banach
couples $(X_0, X_1)$, $(Y_0, Y_1)$, $(Z_0, Z_1)$ and any bilinear operator $T\colon (X_0, X_1) \times (Y_0, Y_1)
\to (Z_0, Z_1)$ such that $T\colon X_0 \times Y_0 \to Z_0$ is compact, we have
\[
T\colon Orb_{[\ao]_{\theta}}^{\ao}(X_0, X_1) \times  Orb_{[\bo]_{\theta}}^{\bo}(Y_0, Y_1)
\to Corb_{[\co]_{\theta}}^{\co}(Z_0, Z_1)
\]
is a compact bilinear operator for every $\theta \in (0, 1)$.
\end{theorem}

\begin{proof}
We apply Theorem \ref{orbcorb}. Observe that for any Banach couple $\ao = (A_0, A_1)$,  we have
$[A_0, A_1]_{\theta} \hookrightarrow (A_0, A_1)_{\theta, \infty}$ for all $\theta \in (0, 1)$
with norm less than or equal to $1$ (see \cite[Theorem 4.7.1]{BL}). This implies that the fundamental
function $\psi_{[\ao]_{\theta}}$ of the space $[\ao]_{\theta}$ satisfies the estimate
\[
\psi_{[\ao]_{\theta}}(s, t) \leq s^{1-\theta}t^{\theta}, \quad\, s, t>0
\]
and so $\psi_{[\ao]_{\theta}} \in \Phi_0$ for all $\theta \in (0, 1)$.

According to multilinear theorem by  Calder\'on (see \cite{Cal} or \cite[Theorem 4.4.1]{BL}), it follows
that, for any bilinear operator $S\colon \ao \times \bo \to \co$, we have
$S\colon [\ao]_{\theta} \times [\bo]_{\theta} \to [\co]_{\theta}$ with
\[
\|S\|_{[\ao]_{\theta} \times [\bo]_{\theta}}
\leq (\|S\|_{A_0 \times B_0 \to C_0})^{1-\theta}(\|S\|_{A_1 \times B_1 \to C_1})^{\theta}.
\]
This implies that
\[
\big([\ao]_{\theta}, [\bo]_{\theta}; [\co]_{\theta}\big) \in \mathcal{B}_{\varphi}(\ao, \bo; \co),
\]
where $\varphi(s, t) = s^{1-\theta} t^{\theta}$ for all $s, t>0.$

Combining the above facts with the well known isometrical formula true for any Banach couple $(X_0, X_1)$,
\[
[X_0, X_1]_{\theta} = [X_0^{\circ}, X_1^{\circ}]_{\theta},
\]
we see that the required result follows from Theorem \ref{orbcorb}.
\end{proof}

Before proceeding applications for bilinear operators on the product of interpolation spaces generated
by Peetre's method $\langle\,\cdot \,\rangle_{\theta}$, we recall that, for any Banach couple $(X_0, X_1)$
and every $\theta \in (0, 1)$, the space $\langle X_0, X_1 \rangle_{\theta}$ is defined as the set of
all elements $x\in X_0 + X_1$ which are represented in the form $x= \sum_{k\in \mathbb{Z}} x_k$
(convergence in $X_0 + X_1$), where the elements $x_k \in X_0 \cap X_1$ are such
that $\sum_{k\in \mathbb{Z}} 2^{-\theta k} x_{k}$ is unconditionally convergent in $X_0$, and
$\sum_{k\in \mathbb{Z}} 2^{(1-\theta) k}x_k$ is unconditionally convergent in $X_1$.
$\langle X_0, X_1 \rangle_{\theta}$ is a~Banach space equipped with the norm
\[
\|x\|_{\langle X_0, X_1\rangle_{\theta}} =
\inf \max_{j=0, 1} \sup \Big\|\sum_{k\in \mathbb{Z}} \varepsilon_k 2^{(1-\theta)k} x_k\Big\|_{X_j},
\]
where the supremum takes over all sequences $(\varepsilon_k) = (\pm 1)$ and the infimum takes over
all representations as above $x=\sum_{k\in \mathbb{Z}} x_k$.

Couples $(c_0, c_0(2^{-n}))$ and $(\ell_1, \ell_1(2^{-n}))$  of $c_0$-spaces and $\ell_1$-spaces modelled
on $\mathbb{Z}$ are denoted by $\vec{c}_0$ and $\vec{\ell}_1$. If $\varphi \in \Phi$, then
$\ell_1(\varphi^{*}(1, 2^{-n})$ is an intermediate space between $\ell_1$ and $\ell_1(2^{-n})$. We denote
by $H_{\varphi}$ the Ovchinnikov functor
\[
Corb_{\ell_1(\varphi^{*}(1, 2^{-n}))}^{\vec{\ell}_1}(\cdot)
\]
If $\varphi(s, t)= s^{1-\theta} t^{\theta}$, for all $s, t>0$ and some $\theta \in (0, 1)$, we write $H_{\theta}$
instead of $H_{\varphi}$.

\begin{theorem}
\label{peetre}
Let $(X_0, X_1)$, $(Y_0, Y_1)$ and $(Z_0, Z_1)$ be Banach couples.~Then, for any bilinear operator
$T\colon (X_0, X_1) \times (Y_0, Y_1) \to (Z_0, Z_1)$, such that $T\colon X_0 \times Y_0 \to Z_0$
is compact, we have
\[
T\colon \langle X_0, X_1\rangle_{\theta} \times \langle Y_0, Y_1\rangle_{\theta} \to H_{\theta}(Z_0, Z_1)
\]
is a compact bilinear operator for every $\theta \in (0, 1)$.
\end{theorem}

\begin{proof}
It is obvious that Banach couples $\vec{c}_0$ and $\vec{\ell}_1$ satisfy approximation property $(Ap)$.
The following well known isometrical formulas
\[
[c_0, c_0(2^{-n})]_{\theta} = c_0(2^{-n\theta}), \quad\, [\ell_1, \ell_1(2^{-n})]_{\theta} = \ell_1(2^{-n\theta})
\]
combined with orbital description of Peetre's functor (see \cite{Janson} or \cite[p.~468]{Ov})
\[
Orb_{c_0(2^{-n\theta})}^{\vec{c}_0}(X_0, X_1) = \langle X_0, X_1\rangle_{\theta}
\]
completes the proof by Theorem \ref{complex} applied for couples $\ao=\bo := \vec{c}_0$ and $\co:= \vec{\ell}_1$.
\end{proof}

We will show applications of the above result to Calder\'on products of Banach function lattices. When the complex method is applied
to a~couple $(X_0, X_1)$ of Banach function lattices, we surmise that $X_j:=X_j(\mathcal{C})$ is a~complexification of $X_j$ for
each $j=0, 1$ on a~$\sigma$-finite complete measure space $(\Omega, \Sigma, \mu)$ with $\text{supp}(X_j)= \Omega$. We recall that
the \emph{Calder\'on product space} $X_{0}^{1-\theta}X_{1}^{\theta}$ is defined for any couple $(X_0, X_1)$ of Banach function
lattices on measure space $(\Omega, \Sigma, \mu)$. It consists of all $f\in L^{0}(\mu)$, such that
$|f| \leq \lambda \,|f_0|^{1-\theta} |f_1|^{\theta}$ $\mu$-a.e.\ for some $\lambda > 0$ and $f_{j}\in X_{j}$ with
$\|f_{j}\|_{X_j} \leq 1$, $j=0, 1$. It is well known (see \cite{Cal}) that $X_{0}^{1-\theta}X_{1}^{\theta}$ is a~Banach function
lattice equipped with the norm
\[
\|f\| = \inf \big\{\lambda >0; \,|f| \leq \lambda \,|f_0|^{1-\theta} |f_1|^{\theta} ,\, \|f_0\|_{X_0} \, \|f_1\|_{X_1}\big\}.
\]
As usual for a given Banach function lattice over $(\Omega, \Sigma, \mu)$, by $X'$, we denote the K\"othe dual space
of $X$ of all $f\in L^0(\mu)$ equipped with the norm
\[
\|f\|_{X'} = \sup_{\|g\|_{X} \leq 1} \int_{\Omega} |fg|\,d\mu.
\]
A Banach function lattice $X$ has the \emph{Fatou property},  provided that the unit ball is closed n $L^0(\mu)$ equipped with
the topology of convergence in measure on $\mu$-finite sets. It is well known that the Fatou property is equivalent to
$X'' = X$, isometrically.

\vspace{2 mm}

Let us draw a useful conclusion in the setting of Calder\'on product spaces.

\begin{corollary}
Let $(X_0, X_1)$, $(Y_0, Y_1)$ and $(Z_0, Z_1)$ be Banach function lattices on the corresponding measure spaces. Assume that
$T\colon (X_0, X_1) \times (Y_0, Y_1) \to (Z_0, Z_1)$ is a~bilinear operator, such that $\colon X_0 \times Y_0 \to Z_0$
is compact. Then,
\[
T\colon (X_0^{1-\theta} X_1^{\theta})^{\circ} \times (Y_0^{1-\theta} Y_1^{\theta})^{\circ} \to
(Z_0^{1-\theta} Z_1^{\theta})''
\]
is a~compact bilinear operator. In particular,
\[
T\colon (X_0^{1-\theta} X_1^{\theta})^{\circ} \times (Y_0^{1-\theta} Y_1^{\theta})^{\circ} \to
Z_0^{1-\theta} Z_1^{\theta}
\]
is compact whenever $Z_0$ and $Z_1$ have the Fatou property.
\end{corollary}

\begin{proof}
For any couple $(E_0, E_1)$ of Banach lattices and $\theta \in (0, 1)$, we have (see \cite[Theorem 2.1]{Nilsson})
\[
\langle E_0, E_1 \rangle_{\theta} = (E_0^{1-\theta} E_1^{\theta})^{\circ}
\]
and (see \cite[Lemma 8.5.1]{Ov})
\[
H_{\theta}(E_0, E_1) \hookrightarrow (E_0^{1-\theta} E_1^{\theta})'' = (E_{0}'')^{1-\theta}(E_{1}'')^{\theta}.
\]
By applying Theorem \ref{peetre}, the required statement is given.
\end{proof}

We conclude with applications to the real methods of interpolation. Let $E$ be a~Banach sequence lattice intermediate with
respect to $(\ell_\infty, \ell_\infty(2^{-n}))$. For a~given Banach couple $\xo$, we denote by $K_E(\xo)$ the $K$-\emph{space}
which is the Banach space of all $x\in X_0 + X_1$ such that $\{K(2^k, x; \xo)\}_{k\in \mathbb{Z}} \in E$ equipped with the
norm
\[
\|x\|_{K_E(\xo)} = \|\{K(2^k, x; \xo)\}\|_{E}.
\]
It is well known that $K_E$ is an exact interpolation functor which is often called $K$-\emph{method} of interpolation.

We also recall the so called $J$-\emph{method} of interpolation. As usual for any Banach couple $\xo = (X_0, X_1)$,
we let $J(t, x; \xo) := \max\{\|x\|_{X_0}, t \|x\|_{X_1}\}$ for any $x\in X_0 \cap X_1$ and all $t>0$. Let $F$ be
a~Banach sequence lattice intermediate with respect to $(\ell_1, \ell_1(2^{-n}))$. By $J_F(\xo)$ we denote the
$J$-\emph{space} which is the Banach space of all $x\in X_0 + X_1$ represented in the form
\[
x = \sum_{k=-\infty}^{\infty} u_k  \quad\, (\text{convergence in $X_0 + X_1$}),
\]
where $\{J(2^k, u_{k}; \xo)\} \in F$ with the norm
\[
\|x\|_{J_F(\xo)} = \inf\Big\{\big\|\big\{J(2^{k}, u_{k}; \xo)\big\}\big\|_{F};\, x = \sum_{k=-\infty}^{\infty} u_k\Big\}.
\]
It is well known that $J_F$ is an exact interpolation functor.

Observe that $\{J(2^k, u_{k}; \xo)\} \in F$ combined with $F \hookrightarrow \ell_1 + \ell_1(2^{-n})$ yields that the
series $\sum_{k=-\infty}^{\infty} u_k$ converges absolutely into $X_0 + X_1$:
\[
\sum_{k=-\infty}^{\infty} \|u_k\|_{X_0 + X_1} \leq \sum_{k=-\infty}^{\infty} J(2^k, u_k; \xo) \min\Big\{1, \frac{1}{2^{k}}\Big\} =
\big\|\big\{J(2^k, u_k; \xo)\big\}\big\|_{\ell_1 + \ell_1(2^{-n})}.
\]

We note that if a~Banach sequence lattice $E$ on $\mathbb{Z}$ satisfies the condition $\ell_\infty \cap \ell_\infty(2^{-n})
\hookrightarrow  E \hookrightarrow \ell_1 + \ell_1(2^{-n})$, then $K_E(\xo)$ and $J_E(\xo)$ are well defined for any Banach couple
$\xo$. This follows immediately from the classical fundamental lemma (see \cite{BL})
\[
K_E(\xo) \hookrightarrow J_E(\xo).
\]
Space $E$ is said to be a parameter of the real method if $K_E(\xo) = J_E(\xo)$ for any Banach couple $\xo$. It is
well known that this is equivalent to the fact that, for any operator $T\colon \vec{\ell}_1 \to \vec{\ell}_\infty$,
one has $T\colon E \to E$ (see, e.g., \cite[Lemma 7.3.1]{Ov}).

We are now able to state our general bilinear interpolation theorem on compactness for bilinear operators
on real methods spaces.

\begin{theorem}
\label{realmethod}
Let $E_0$, $E_1$ and $F$ be Banach sequence lattices,  such that $(E_0, E_1; F) \in \mathcal{B}_{\varphi}(\vec{\ell}_1,
\vec{\ell}_{1}; \vec{\ell}_{\infty})$ for some $\varphi \in \Phi_0$, and let $\psi_{E_0}(s, 1)\to 0$, $\psi_{E_1}(s, 1)\to 0$
and $\phi_F(s, 1) \to 0$ as $s\to 0$. Then, for any Banach couples $(X_0, X_1)$, $(Y_0, Y_1)$ and $(Z_0, Z_1)$
and any bilinear operator $T\colon (X_0, X_1) \times (Y_0, Y_1) \to (Z_0, Z_1)$ such that $T\colon X_0 \times Y_0 \to Z_0$ is
compact, we obtain
\[
T\colon J_{E_0}(X_0, X_1) \times J_{E_1}(Y_0, Y_1) \to K_F(Z_0, Z_1)
\]
is a compact bilinear operator.
\end{theorem}

\begin{proof}
From the well known isometrical description of coorbital (resp., orbital) of the $K$-space (resp., $J$-space), we have,
for any Banach couple $(A_0, A_1)$ (see \cite[Theorems 3.3.4, 3.4.12]{BK} or \cite[Theorems 7.1.1, 7.2.1]{Ov}):
\begin{align*}
K_{F}(A_0, A_1) = Corb_{F}^{\vec{\ell}_\infty}(A_0, A_1) \quad\, (\text{resp., $J_E(A_0, A_1) = Orb_{E}^{\vec{\ell}_1}(A_0, A_1)$}).
\end{align*}
Since $\vec{\ell}_1$ is a regular couple,
\[
J_E(A_0, A_1) = J_E(A_{0}^{\circ}, A_{1}^{\circ}).
\]
Now we are in a position to apply Theorem \ref{orbcorb} to get the statement.
\end{proof}

We provide a result which gives a complete description of triples of Banach sequence lattices
$(E_0, E_1; F) \in \mathcal{B}_{\varphi}(\vec{\ell}_1, \vec{\ell}_{1}; \vec{\ell}_{\infty})$ in terms
of boundedness of the convolution operator $\sigma$ defined on $(\ell_1 + \ell_1(2^{-n}))\times
(\ell_1 + \ell_1(2^{-n}))$ by $\sigma (x,y) = x\star y$, for all $x = \{x_n\}$ and $y = \{y_n\}$ in
$\ell_1 + \ell_1(2^{-n})$, where
\[
x\star y := \bigg\{\sum_{m=-\infty}^{\infty} x_m y_{k-m}\bigg\}_{k=-\infty}^{\infty}.
\]
If Banach sequence lattices  $E_0$, $E_1$ and $E_2$ intermediate with respect to $\vec{\ell}_1$ are such
that the convolution operator $\sigma \colon E_0 \times E_1 \to E_2$, then we write $E_0 \star E_1 \subset
E_2$ for short.

\vspace{2 mm}

At first we prove the following lemma.

\begin{theorem}
\label{conv}
Let Banach sequence lattices $E_0$, $E_1$ and $E_2$ be intermediate with respect to $\vec{\ell}_1$ such that
$E_0 \star E_1 \subset E_2$. Then for $F=J_{E_2}(\vec{\ell}_{\infty})$, we have $(E_0, E_1; F) \in \mathcal{B}(\vec{\ell}_1, \vec{\ell}_1; \vec{\ell}_{\infty})$. In particular $(E_0, E_1; F)\in \mathcal{B}_{\varphi}(\vec{\ell}_1, \vec{\ell}_1; \vec{\ell}_{\infty})$
with $\varphi$ defined by
\[
\varphi(s, t):= \sup \|T\|_{E_0 \times E_1\to F}, \quad\, s, t>0,
\]
where the supremum takes over all bilinear operators $T\colon \vec{\ell}_1 \times \vec{\ell}_1\to \vec{\ell}_{\infty}$,
such that $\|T\|_{\ell_1 \times \ell_1 \to \ell_{\infty}} \linebreak \leq s$ and
$\|T\|_{\ell_1(2^{-n}) \times \ell_1(2^{-n}) \to \ell_{\infty}(2^{-n})} \leq t$.
\end{theorem}

\begin{proof}
Let us assume that $E_0 \star E_1 \subset E_2$ and let $T\colon \vec{\ell}_1 \times \vec{\ell}_1 \to \vec{\ell}_{\infty}$ be any bilinear
operator with norm less than or equal to $1$.

Fix $x = \{x_m\} \in E_0$ and $y\in \{x_k\} \in E_1$. By $E_j \hookrightarrow \ell_1 + \ell_1(2^{-n})$ for $j\in \{0, 1\}$,
then the two series
\[
x = \sum_{m=-\infty}^{\infty} x_m e_m, \quad\, y = \sum_{k=-\infty}^{\infty} y_k e_k \quad\,
\]
converge absolutely in $\ell_1 + \ell_1(2^{-n}))$, where $e_n$ denotes the standard unit basis vector for each $n\in \mathbb{Z}$.

Since $T\colon (\ell_1 + \ell_1(2^{-n})) \times (\ell_1 + \ell_1(2^{-n})) \to \ell_{\infty} + \ell_{\infty}(2^{-n})$ is continuous,
\begin{align*}
T(x, y) = \sum_{m, k=-\infty}^{\infty} T(x_m e_m, y_k e_k) = \sum_{m, k=-\infty}^{\infty} T(x_m e_m, y_{k-m} e_{k-m})
\end{align*}
where each double series converges absolutely into $\ell_{\infty} + \ell_{\infty}(2^{-n})$. Consequently,
\[
T(x, y) = \sum_{k=-\infty}^{\infty}\bigg(\sum_{m=-\infty}^{\infty} T(x_m e_m, y_{k-m} e_{k-m})\bigg)
\]
with convergence in $\ell_{\infty} + \ell_{\infty}(2^{-n})$.

Observe that for each $k\in \mathbb{Z}$, we have (by $\|T\|_{\vec{\ell}_1 \times \vec{\ell}_1 \to \vec{\ell}_{\infty}} \leq 1$)
\begin{align*}
\Big\|\sum_{m=-\infty}^{\infty} T(x_m e_m, y_{k-m} e_{k-m})\Big\|_{\ell_{\infty}}  & \leq \sum_{m=-\infty}^{\infty} \|T(x_m e_m, y_{k-m} e_{k-m})\|_{\ell_{\infty}} \\
& \leq  \sum_{m=-\infty}^{\infty} \|x_m e_m\|_{\ell_1}\,\|y_{k-m} e_{k-m}\|_{\ell_1} \\
& \leq \sum_{m=-\infty}^{\infty} |x_m| |y_{k-m}| = (|x|\star |y|)_{k}
\end{align*}
and similarly for each $k\in \mathbb{Z}$,
\begin{align*}
\Big\|\sum_{m=-\infty}^{\infty} T(x_m e_m, y_{k-m} e_{k-m})\Big\|_{\ell_{\infty}(2^{-n})} & \leq
\sum_{m=-\infty}^{\infty} \|T(x_m e_m, y_{k-m} e_{k-m})\|_{\ell_{\infty}(2^{-n})} \\
& \leq \sum_{m=-\infty}^{\infty} \|x_m e_m\|_{\ell_1(2^{-n})}\,\|y_{k-m} e_{k-m}\|_{\ell_1(2^{-n})} \\
& \leq 2^{-k}\,\sum_{m=-\infty}^{\infty} |x_m| |y_{k-m}| = (|x|\star |y|)_{k}.
\end{align*}
Combining the above estimates, we conclude that
\[
u_k = \sum_{m=-\infty}^{\infty} T(x_m e_m, y_{k-m} e_{k-m}) \in \ell_{\infty} \cap \ell_{\infty}(2^{-n})
\]
and
\[
\big\{J(2^k, u_k; \ell_{\infty})\big\}_{k=-\infty}^{\infty} \leq \big\{(|x|\star |y|)_{k}\big\}_{k=-\infty}^{\infty} = |x| \star |y|.
\]
Since
\[
T(x, y) = \sum_{k=-\infty}^{\infty} u_k \quad\, (\text{convergence in $\ell_{\infty} + \ell_{\infty}(2^{-n})$})
\]
and there exists a positive constant $C$ (since $\sigma$ is positive, $\sigma \colon E_0 \times E_1 \to E_2$
is a bounded bilinear operator)
\[
\||x|\star |y|\|_{E_2}  \leq C \|x\|_{E_0}\,\|y\|_{E_1}, \quad\, (x, y) \in E_0 \times E_{1},
\]
we get that $T(x, y) \in J_{E_2}(\vec{\ell}_{\infty})$ with
\begin{align*}
\|T(x, y)\|_{J_{E_2}(\vec{\ell}_{\infty})} & \leq \|\{J(2^{k}, u_k; \vec{\ell}_{\infty})\}\|_{E_2} \leq C \|x\|_{E_0}\,\|y\|_{E_1}.
\end{align*}
This completes the proof of the first statement. The second statement is obvious.
\end{proof}

Let us conclude by remarking that the convolution operator $\sigma \in \mathcal{B}(\vec{\ell}_1, \vec{\ell}_{1}; \vec{\ell}_{\infty})$,
and so an immediate consequence of Theorem \ref{conv}, is the following result. If $E_0$, $E_1$ and $F$ are Banach sequence lattices
intermediate with respect to $\vec{\ell}_1$ and if $F$ is a~real parameter of the real method, then $(E_0, E_1; F) \in
\mathcal{B}(\vec{\ell}_1, \vec{\ell}_1; \vec{\ell}_{\infty})$, if and only if, $E_0 \star E_1 \subset F$. This observation in combination
with Theorem \ref{realmethod} in particular yields a more general variant of a~bilinear compactness interpolation theorem established
in \cite[Theorem 3.1]{Luz} for spaces generated by parameters of the real method.

\vspace{3 mm}

\bibliographystyle{amsplain}

\vspace{3 mm}

\noindent Mieczys{\l}aw Masty{\l}o \\
Faculty of Mathematics \& Comp. Sci. \\
Adam~Mickiewicz University in Pozna{\'n}\\
Umultowska 87 \\
61-614 Pozna{\'n}, Poland

\vspace{1 mm}

\noindent
E-mail: \,{\tt mastylo$@$amu.edu.pl} \\

\noindent
Eduardo B.~Silva  \\
Departamento de Matem\'atica \\
Universidade Estadual de Maring\'a--UEM \\
Av.~Colombo 5790 \\
Maring\'a - PR  \\
Brazil - 870300-110

\vspace{1 mm}

\noindent E-mail: \,{\tt ebsilva@uem.br}
\end{document}